\documentclass[12pt]{amsart}
\usepackage{geometry}                
\usepackage{graphicx}
\usepackage{amssymb}
\usepackage{epstopdf}
\usepackage{amsmath}
\usepackage{color}
\usepackage{hyperref}
\usepackage{pdfsync}
 \usepackage{tikz}
 \usepackage{tikz-cd}
\usepackage{datetime2}

\usepackage[all]{xy}
\xyoption{matrix}
\xyoption{arrow}

\usetikzlibrary{arrows,decorations.pathmorphing,backgrounds,positioning,fit,petri}

 \tikzset{help lines/.style={step=#1cm,very thin, color=gray},
help lines/.default=.5} 
\tikzset{thick grid/.style={step=#1cm,thick, color=gray},
thick grid/.default=1} 

\newtheorem{thm}{Theorem}[section]
\newtheorem{lem}[thm]{Lemma}
\newtheorem{cor}[thm]{Corollary}
\newtheorem{prop}[thm]{Proposition}

 \newtheorem{innercustomthm}{Theorem}
\newenvironment{customthm}[1]
 {\renewcommand\theinnercustomthm{#1}\innercustomthm}{\endinnercustomthm}

\theoremstyle{definition}
\newtheorem{defn}[thm]{Definition}
\newtheorem{eg}[thm]{Example}

\theoremstyle{remark}

\numberwithin{equation}{section}

\DeclareMathOperator{\Hom}{Hom}%
\DeclareMathOperator{\Ext}{Ext}%
\DeclareMathOperator{\End}{End}%
\DeclareMathOperator{\undim}{\underline{dim}}
 
\newcommand{\field}[1]{\mathbb{#1}}
\newcommand{\ZZ}{\ensuremath{{\field{Z}}}}

\newcommand{\RR}{\ensuremath{{\field{R}}}}

\newcommand{\NN}{\ensuremath{{\field{N}}}}

\newcommand{\commentout}[1]{}

\newcommand{\cC}{\ensuremath{{\mathcal{C}}}}
\newcommand{\cD}{\ensuremath{{\mathcal{D}}}}

\newcommand{\cF}{\ensuremath{{\mathcal{F}}}}

\newcommand{\cI}{\ensuremath{{\mathcal{I}}}}

\newcommand{\cS}{\ensuremath{{\mathcal{S}}}}

\newcommand{\cU}{\ensuremath{{\mathcal{U}}}}

\newcommand{\cW}{\ensuremath{{\mathcal{W}}}}

\newcommand{\cX}{\ensuremath{{\mathcal{X}}}}

\title{Which cluster morphism categories are CAT(0)}
\author{Kiyoshi Igusa}
\address{Department of Mathematics, Brandeis University, Waltham, MA 02454}\email{igusa@brandeis.edu}
\thanks{First author supported by Simons Foundation Grant \#686616}
\author{Gordana Todorov}
\address{Department of Mathematics, Northeastern University, Boston, MA 02115}\email{g.todorov@northeastern.edu}

\subjclass[2020]{
16G20; 20F55}

\begin{document}

\begin{abstract}
The cluster morphism category of an hereditary algebra was introduced in \cite{IT13} to show that the picture space of an hereditary algebra of finite representation type is a $K(\pi,1)$ for the associated picture group, thereby allowing for the computation of the homology of picture groups of finite type as carried out in \cite{ITW} for the case of $A_n$.

In this paper we show that the cluster morphism category is a $CAT(0)$-category for hereditary algebras of finite or tame type with only small tubes. As a consequence, we get that the classifying space of the cluster morphism category is a locally $CAT(0)$ space and, as a consequence of that, we get that this classifying space is a $K(\pi,1)$.
\end{abstract}

\maketitle


\def\gg{\gamma}


%
%

\section*{Introduction}

We define a $CAT(0)$-category to be a cubical category (Definition \ref{def: cubical category}) whose classifying space is locally $CAT(0)$. 
To show that the cluster morphism category of an hereditary algebra (Definition \ref{def: cluster morphism category}) is $CAT(0)$ we first show it is a cubical category.

\begin{customthm}{A}[Theorem \ref{thm: CMC is cubical}]\label{thm A}
The cluster morphism category of any hereditary algebra $\Lambda$ is a cubical category.
\end{customthm}

In Theorem \ref{thm: conditions for CAT0} we recall (from \cite{NonX}) criteria for when a cubical category is a $CAT(0)$-category. The rest of the paper is devoted to verifying these categorical conditions on the cluster morphism categories for the algebras in our main theorem:

\begin{customthm}{B}[Theorem \ref{main thm}]\label{thm B}
The cluster morphism category of an hereditary algebra of finite or tame representation type is a $CAT(0)$-category if and only if there are no tubes of rank $\ge3$ in its Auslander-Reiten quiver.
\end{customthm}

For hereditary algebras of finite type, we already showed in \cite{IT13} that the classifying space of the cluster morphism category of $\Lambda$ is a $K(\pi,1)$ for the picture group of $\Lambda$. We also showed, in \cite{IT13}, that this classifying space is homeomorphic to the ``picture space'' of $\Lambda$ defined in \cite{ITW}. The main result of \cite{ITW} was the calculation of the cohomology of the picture group for $\Lambda$ of type $A_n$, assuming that the picture space is a $K(\pi,1)$ for the picture group of $\Lambda$. In this paper we show that the classifying space of the cluster morphism category, for $\Lambda$ of finite type or tame with small tubes, is locally $CAT(0)$ by showing that those categories are $CAT(0)$-categories (Theorem \ref{thm B}). This implies that the classifying space is a $K(\pi,1)$, but being locally $CAT(0)$ is a stronger statement.

The cluster morphism category of an hereditary algebra $\Lambda$ has, as objects, the finitely generated wide subcategories $\cW$ of $mod\text-\Lambda$. A morphism $[T]:\cW\to \cW'$ is given by a partial cluster $T$ in $\cW$ so that $T^\perp\cap \cW=\cW'$ (Definition \ref{def: cluster morphism category}.) When $\Lambda$ has finite type, this is a finite category (having a finite number of objects and morphisms). Thus, its classifying space is a finite cell complex. For any $\Lambda$, this classifying space is finite dimensional since at most $n$ nonidentity morphisms can be composed. This implies, e.g., that the fundamental group of the cluster morphism category, which is the picture group of $\Lambda$, is torsion-free.

In Section \ref{section one: cubical categories} we recall the definition of a cubical category (Definition \ref{def: cubical category}) and show that the cluster morphism category of any hereditary algebra is cubical (Theorem \ref{thm: CMC is cubical}). The main property of cubical properties is that every morphism is embedded in a cube of morphisms making its classifying space a union of cubes. In the case of the cluster morphism category, these cubes are given by factoring each cluster morphism of rank $n$ into $n!$ signed exceptional sequences.

In Theorem \ref{thm: conditions for CAT0} we give conditions under which a cubical category is a $CAT(0)$-category and we use the rest of this paper to determine when these conditions hold for cluster morphism categories of hereditary algebras of finite or tame type. 


%
%

\section{Cubical categories}\label{section one: cubical categories}

We review the definition of a cubical category from \cite{NonX}. To simplify the discussion, we add the condition that the category has a ``rank'' filtration, a condition satisfied by the cluster morphism category.

The basic example of a cubical category is the \emph{$n$-cube category} $\cI^n$ whose objects are the $2^n$ subsets of the set $[n]=\{1,2,\cdots,n\}$ and whose morphisms are the inclusion maps. The ``rank'' of a subset of $[n]$ is defined to be the size of its complement. This is a finite category with at most one morphism between any two objects. Given any morphism $f:X\to Y$ in any category $\cC$, we denote by $\cF ac(f)$ the category of factorizations $X\to Z\to Y$ of $f$. A morphism of factorizations $(X\to Z\to Y)\to (X\to Z'\to Y)$ is a morphism $Z\to Z'$ making the 4 object diagram commute. $\cF ac(f)$ has initial object $X\xrightarrow=X\to Y$ and terminal object $X\to Y\xrightarrow=Y$. We have a \emph{middle object functor} $\cF ac(f)\to \cC$ sending $(X\to Z\to Y)$ to $Z$. 

\begin{defn}\label{def: cubical category}
A \emph{cubical category} is a small category with graded object set $\cC=\coprod_{n\in \ZZ} \cC_n$ having the following properties where objects $X\in \cC_n$ are said to have \emph{rank} $rk\,X=n$ and a morphism $f:X\to Y$ is said to have rank $rk\,f=rk\,X-rk\,Y$.
\begin{enumerate}
\item Every morphism $f:X\to Y$ has $rk\,f\ge0$ and $rk\,f=0$ only when $X=Y$ and $f=id_X$.
\item For every morphism $f:X\to Y$, the category $\cF ac(f)$ of factorizations of $f$ is isomorphic to the $n$-cube category $\cI^n$ with $n=rk\,f$ and the ``middle object functor'': $\cF ac(f)\to \cC$ is an embedding.
\item A morphism $f:X\to Y$ of rank $n$ is uniquely determined by its $n$ first factors. These are the rank $1$ morphisms $t_i:X\to A_i$, $i=1,\cdots,n$, which can be completed to a factorization $X\to A_i\to Y$ of $f$.
\item A morphism $f:X\to Y$ of rank $n$ is uniquely determined by its $n$ last factors. These are the rank $1$ morphisms $x_i:B_i\to Y$, $i=1,\cdots,n$, which can be completed to a factorization $X\to B_i\to Y$ of $f$.
\end{enumerate}
\end{defn} 

\begin{prop} Let $f:X\to Y$ be a morphism in a cubical category $\cC$. Then, the functor $\cI^n\to \cC$ given by composing any isomorphism $\cI^n\cong \cF ac(f)$ with the embedding $\cF ac(f)\to \cC$, as given in Condition (2) above, preserves ranks of morphisms but increases the rank of each object by $rk\,Y$.
\end{prop}

\begin{proof}
Since $X,Y$ are initial and terminal in $\cF ac(f)$, they correspond to the initial and terminal objects of $\cI^n$, call them $A,B$. Then every object $C$ in $\cI^n$ lies in a chain of $n$ rank 1 morphism $A\to C_{n-1}\to\cdots\to C_2\to C_{1}\to B$ (so that $rk\,C_j=j$. Since $\cI^n\cong \cF ac(f)$, every object $Z$ in $\cF ac(f)$ lies in a chain of $n$ nonidentity morphisms from $X$ to $Y$. Since $rk\,X-rk\,Y=n$, all of these nonidentity morphisms must have rank 1 and each object $C_j$ in the chain of morphisms in $\cI^n$ corresponds to an object of $\cF ac(f)$ of rank $rk\,Y+j$. Since all objects of $\cI^n$ lie in such a chain, the isomorphism $\cI^n\cong \cF ac(f)$ increases the rank of each object by $rk\,Y$. It follows that corresponding morphisms have the same rank.
\end{proof}

We review some basic cluster theory for hereditary algebras which is needed for the definition of a cluster morphism category (Definition \ref{def: cluster morphism category}). A module $E$ is called \emph{exceptional} if it is a rigid brick where \emph{rigid} means $\Ext_\Lambda^1(E,E)=0$ and a \emph{brick} is a module whose endomorphism ring is a division algebra. An exceptional module is uniquely determined by its dimension vector $\undim E\in \NN^n$ where $n$ is the \emph{rank} of $\Lambda$, i.e., the number of nonisomorphic simple $\Lambda$-modules. A \emph{cluster-tilting object} for $\Lambda$ is a rigid object in the cluster category of $\Lambda$ with $n$ components \cite{BMRRT}. This is equivalent to a set of exceptional modules $E_1,\cdots,E_k$ and shifted indecomposable projective modules $P_{k+1}[1],\cdots,P_n[1]$ so that $\Ext^1_\Lambda(E,E)=0$ and $\Hom_\Lambda(P,E)=0$ for $E=\bigoplus E_i$ and $P=\bigoplus P_j$. We note that every shifted projective object $P[1]$ is uniquely determined by its dimension vector $\undim(P[1]):=-\undim P\in \ZZ^n$. A \emph{partial cluster-tilting set} is any subset of the set 
\[
	\{E_1,\cdots,E_k,P_{k+1}[1],\cdots,P_n[1]\}.
\]
These are objects of $\cD^b(\Lambda)$, the bounded derived category of $mod\text-\Lambda$, which are denoted $T_1,\cdots,T_n$. Thus $T=\bigoplus T_i=E\oplus P[1]\in \cD^b(\Lambda)$. We use the notation $|T_i|$ for the underlying module of $T_i$. Thus $|E_i|=E_i$ and $|P_j[1]|=P_j$. Also, $|T|=\bigoplus |T_i|$.

We denote by $\cC_\Lambda$ the set of isomorphism classes of exceptional modules $E$ and shifted indecomposable projective modules $P[1]$. Thus, every partial cluster-tilting set is a subset of $\cC_\Lambda$. Given any partial cluster tilting object $T$, its \emph{(right $\Hom$-$\Ext$) perpendicular category} $T^\perp$ ($=|T|^\perp$) is the full subcategory of $mod\text-\Lambda$ of all modules $M$ so that $\Hom_\Lambda(|T|,M)=0=\Ext^1_\Lambda(|T|,M)$. Then $T^\perp$ is a finitely generated {wide subcategory} of $mod\text-\Lambda$. We recall that a full subcategory $\cW$ of $mod\text-\Lambda$ is \emph{wide} if it is an exactly embedded abelian subcategory of $mod\text-\Lambda$ which is closed under extensions. A wide subcategory $\cW$ is \emph{finitely generated} if it contains one object $P$ so that every other object of $\cW$ is a quotient of $P^k$ for some $k$. Every finitely generated wide subcategory of $mod\text-\Lambda$ occurs as a perpendicular category (\cite{S92}, \cite{IOTW3}). Also, every finitely generated wide subcategory $\cW$ is isomorphic to $mod\text-H_\cW$ for some hereditary algebra $H_\cW$. A cluster-tilting set in $\cW$ is defined to be a subset of $\cD^b(\cW)$ corresponding to a cluster-tilting set for $H_\cW$. We denote by $\cC_\cW$ the set of isomorphism classes of exceptional objects of $\cW$ and shifts $P[1]\in \cD^b(\cW)$ of indecomposable projective objects $P\in\cW$. Thus $\cC_\cW\cong\cC_{H_\cW}$.

\begin{defn}\label{def: cluster morphism category}
The \emph{cluster morphism category} $\cX(\Lambda)$ of an hereditary algebra $\Lambda$ has as objects the finitely generated wide subcategories $\cW$ of $mod\text-\Lambda$. A morphism $[T]:\cW\to \cW'$ in $\cX(\Lambda)$ is defined to be a \emph{partial cluster tilting set} $T$ in $\cW$ so that $\cW'=T^\perp\cap \cW$. 

The composition of cluster morphisms 
\[
	\cW\xrightarrow{[S]} \cW''\xrightarrow{[T]} \cW'
\]
is defined to be the cluster morphism $[R]:\cW\to \cW''$ where $R$ is the partial cluster-tilting object in $\cW$ given by
\[
	R=S\,\textstyle{\coprod}\, \sigma_S(T)
\]
where $\sigma_S:\cC_{\cW''}\to \cC_\cW$ sends $T\in \cC_{\cW''}$ to the unique element $\sigma_S(T)\in\cC_\cW$ satisfying the following.
\begin{enumerate}
\item $S\,\textstyle{\coprod}\, \sigma_S(T)
=\{S_1,\cdots,S_r,\sigma_ST\}$ is a partial cluster tilting set for $\cW$.
\item $\undim \sigma_ST-\undim T$ is an integer linear combination of $\undim S_i$.
\item $T^\perp\cap \cW''=(S_1\oplus \cdots\oplus S_r\oplus \sigma_ST)^\perp \cap \cW$.
\end{enumerate}
The mapping $\sigma_S(T)$ is a bijection from $\cC_{\cW''}$ to the set of all objects of $\cC_{\cW}$ which are compatible with $S$ \cite[Prop. 1.8]{IT13}.
\end{defn}

\begin{thm}[Theorem \ref{thm A}]\label{thm: CMC is cubical}
The cluster morphism category of any hereditary algebra $\Lambda$ is a cubical category.
\end{thm}

\begin{proof}
Let $\cC$ denote the cluster morphism category of $\Lambda$. The objects of $\cC$ are finitely generated wide subcategories $\cW$ of $mod\text-\Lambda$. Each of these is equivalent to a module category of an hereditary algebra $H_\cW$ of finite rank which we define to be the rank of $\cW$. Also, $\cW$ is the perpendicular category of some partial cluster tilting set $T=\{T_1,\cdots,T_t\}$, i.e., $\cW=T^\perp$ and $rk\,\cW=n-t$.

A cluster morphism $[R]:\cW\to \cW'$ is, by definition, a partial cluster-tilting set in $\cW$: 
\[
	R=\{R_1,\cdots,R_r\}\subset \cC_\cW
\]
so that $\cW'=\cW\cap R^\perp$. It follows that $r=rk\,\cW-rk\,\cW'$. When $r=0$, $\cW'=\cW$ and $[\emptyset]:\cW\to\cW$ is the identity map. Thus, objects and morphisms of $\cC$ have ranks so that Condition (1) in Definition \ref{def: cubical category} is satisfied.

Every factorization of a cluster morphism $[R]:\cW\to \cW'$ of rank $r$ is given by
\[
	\cW\xrightarrow{[S]} \cW''\xrightarrow{[T]} \cW'
\]
where $S$ is a subset of $R$ of size, say $s$ and $T\subset \cC_{\cW''}$ is a partical cluster tilting set in $\cW''=S^\perp\cap \cW$ so that
\[
	R=S\,\textstyle{\coprod}\, \sigma_S(T)
\]
where $\sigma_S:\cC_{\cW''}\to \cC_\cW$ is described in Definition \ref{def: cluster morphism category}. The objects $\cW''=S^\perp\cap \cW$ are distinct for different $S$ since, otherwise we would have $(S')^\perp=S^\perp=(S\cup S')^\perp$ which is impossible since $S\cup S'$ is larger than $S$. (The size of $S$ is $rk\,\cW''-rk\,\cW'$.) This gives an isomorphism between the category of factorizations of the cluster morphism $[R]:\cW\to \cW'$ with the cube category $\cI^r$. So, Condition (2) is satisfied. 

The possible first factors of the cluster morphism $[R]$ are the elements $[R_i]$ of $R$. So, $[R]$ is determined by its first factors and Condition (3) is satisfied.

To verify the last property (4), let $[L_i]:\cW_i\to \cW'$ be the last factors of $[R]$. Then we have the following factorizations of the morphism $[R]$.
\[
	\cW\xrightarrow{[S^i]}\cW_i\xrightarrow{[L_i]}\cW'.
\]
where $[S^i]=[R_1,\cdots,\widehat{R_i},\cdots,R_r]$. Since $\cW_i=\cW\cap (S^i)^\perp$, we have the following linear equation
\[
	\left< \undim R_i,\undim L_j\right>=0
\]
for $j\neq i$ where $\left<\cdot,\cdot\right>$ is the \emph{Euler-Ringel form} uniquely determined by the formula
\[
	\left< \undim A,\undim B\right>=\dim_K\Hom_\Lambda(A,B)-\dim_K\Ext^1_\Lambda(A,B)
\]
and $\undim M\in \NN^n$ is the dimension vector of any $\Lambda$-module $M$. If we choose a cluster tilting set $T=\bigoplus T_k$ for $\cW'=S^\perp\cap \cW$, we also have
\[
	\left< \undim R_i,\undim T_k\right>=0
\]
for all $i,k$. Since $\sigma_{S^i}(T_i)=R_i$, the characterizing properties of the mapping $\sigma_{S^i}:\cC_{\cW''}\to \cC_\cW$ also give one more linear equation:
\[
	\left< \undim R_i,\undim L_i\right>=\left< \undim L_i,\undim L_i\right>. 
\]
This is the dimension over $K$ of the division algebra $\End_\Lambda(L_i)$. The wide subcategory $\cW$ is determined by the last factors $L_i:\cW_i\to \cW'$ since $\cW$ is generated by the union of the $\cW_i$. The rank of $\cW$ is equal to $r$ plus the rank of $\cW'$.
Since the Euler-Ringel form, restricted to $\cW$ is nonsingular, the linear equations above have a unique solution for $\undim R_i$ given the last factors $L_j$. Each $R_i\in\cW\coprod \cW[1]$ is determined by its dimension vector $\undim R_i\in\ZZ^n$. Therefore, the last factors $L_i$ determine the first factors $R_i$ of the cluster morphism $[R]$. And the first factors determine $[R]$. Thus, Condition (4) is satisfied. So, $\cC$, the cluster morphism category of $\Lambda$ is a cubical category.
\end{proof}

%
%

\section{Pairwise compatibility}

Let $\cC$ be a cubical category (Definition \ref{def: cubical category}). We say that $\cC$ is a \emph{$CAT(0)$-category} if its classifying space is locally $CAT(0)$. The following Theorem \ref{thm: conditions for CAT0} from \cite{NonX} converts this into a sufficient algebraic condition. This is the theorem that we use in order to determine when a cluster morphism category is $CAT(0)$.

We recall that the classifying space of a cubical category is a locally cubical space in the sense that its universal covering is cubical. Gromov \cite{Gromov87} gave necessary and sufficient conditions for a simply connected cubical space to be $CAT(0)$. In \cite{NonX} we used Gromov's theorem to prove Theorem \ref{thm: conditions for CAT0} below.

\begin{thm}\cite{NonX}\label{thm: conditions for CAT0}
The following are sufficient conditions for a cubical category $\cC$ to be a $CAT(0)$-category. 
\begin{enumerate}
\item $n$ morphisms $X\to A_i$ form the first factors of a morphism $X\to Y$ of rank $n$ if and only if every pair of them form the first factors of a morphism of rank $2$. 
\item $n$ morphisms $B_i\to Y$ form the last factors of a morphism $X\to Y$ of rank $n$ if and only if every pair of them form the last factors of a morphism of rank $2$.
\item There exists a faithful functor from $\cC$ to a group $G$ considered as a category with one object.
\end{enumerate}
Furthermore, the first two conditions are necessary.
\end{thm}

It is immediate that the cluster morphism category satisfies the first condition since the objects $T_i\in\cC(\cW)$ in a morphism $[T]:\cW\to\cW'$ are given by the pairwise compatibility condition that they do not extend each other. In this section we will verify the second condition for $\Lambda$ of finite type and of tame type with only small tubes. For this we use the following result of Speyer and Thomas.

\begin{thm}\cite{ST}\label{thm: ST}
The dimension vectors $\undim X_i\in\ZZ^n$ of a collection of $n$ exceptional modules and shifted exceptional modules $X_i$ (with underlying modules $M_i$) form the $c$-vectors of a cluster tilting object for an hereditary algebra $\Lambda$ if and only if the following two conditions hold.
\begin{enumerate}
\item Objects $X_i,X_j$ of the same sign are $\Hom$-orthogonal.
\item The modules $M_i$ form an exceptional sequence with the negative modules to the right of the positive modules.
\end{enumerate}
\end{thm}

The second conditions is equivalent to the following three conditions.
\begin{enumerate}
\item The positive modules $X_i$ can be arranged in an exceptional sequence.
\item The underlying modules $M_j$ of the shifted objects $X_j=M_j[1]$ can be arranged in an exceptional sequence.
\item For every $X_i=M_i$ and $X_j=M_j[1]$, $(M_i,M_j)$ forms an exceptional pair. 
\end{enumerate}

\begin{cor}\label{hom orth}
$c$-vectors for $\Lambda$ are given by a pairwise compatibility condition if and only if the following condition is satisfied:

$(\ast)$ Every collection of $\Hom$-orthogonal exceptional modules having the property that any pair is an exceptional pair can be arranged in an exceptional sequence.
\end{cor}

\begin{cor}
The cluster morphism category of $\Lambda$ satisfies the second condition in Theorem \ref{thm: conditions for CAT0} if an only if it satisfies Condition $(\ast)$ in Corollary \ref{hom orth} holds.
\end{cor}

\begin{proof}
The linear equations relating first factors $R_i$ and last factors $L_j$ of a cluster morphism $R=[R_1,\cdots,R_k]$ were given in the proof of Theorem \ref{thm: CMC is cubical}:
\[
	\left< \undim R_i,\undim L_j\right>= \begin{cases} \dim_K\End(L_i) & \text{if } i=j\\
   0 & \text{otherwise}
    \end{cases}
\]
By \cite[Theorem 4.3.1]{IOTW3} these equations imply that the last factors of any cluster morphism $[R]:mod\text-\Lambda\to 0$ have, as dimension vectors, the negatives of the $c$-vectors of the cluster tilting object $R$ of $\Lambda$. By the previous Corollary \ref{hom orth}, Condition $(\ast)$ is necessary for these last factors to be given by a pairwise compatibility condition as required by (2) in Theorem \ref{thm: conditions for CAT0}. 

To show sufficiency, suppose that $(\ast)$ holds and $[L_i]:\cW_i\to \cW$ is a collection of rank 1 cluster morphisms which are pairwise compatible in the sense that any pair $[L_i],[L_j]$ form the last factors of a rank 2 cluster morphism $[L_{ij}]:\cW_{ij}\to \cW$. Then each object $L_i$ is left perpendicular to $\cW$, i.e., lies in the wide subcategory $\,^\perp\cW$ and the same objects $L_i$ give rank 1 morphisms $[L_i]:\cW_i\cap\,^\perp\cW\to0$ which are pairwise compatible. Thus, by $(\ast)$ they form the last factors of some morphism $[R]:\cW'\to 0$ where $\cW'\subset \,^\perp\cW$. Equivalently,
\[
	\left<\undim R_i,\undim L_j\right>=\delta_{ij} \left<\undim L_j,\undim L_j\right>
\]
which implies that $[R]$ gives a morphism $\cW''\to \cW$ with last factors $[L_i]:\cW_i\to \cW$. Thus condition (2) in Theorem \ref{thm: conditions for CAT0} holds.
\end{proof}

It remains to determine when Condition $(\ast)$ holds.

\begin{thm}\label{old thm B}
Let $\Lambda$ be an hereditary algebra which is either of finite or tame representation type. Then Condition $(\ast)$ in Corollary \ref{hom orth} holds if and only if the Auslander-Reiten quiver of $\Lambda$ has no tubes of rank $\ge3$.
\end{thm}

\begin{proof}
If $\Lambda$ has a tube of rank $\ge3$ then $(\ast)$ fails since the quasi-simple objects of the mouth of this tube are $\Hom$ orthogonal and any pair of them forms an exceptional sequence, but, taken together, they do not form an exceptional sequence since they extend each other in a cycle.

Conversely suppose that $\Lambda$ is either of finite type or tame with only small tubes (of rank $\le2$). Let $M_i$ be a collection of $\Hom$-orthogonal exceptional $\Lambda$ modules so that any two form an exceptional pair. Then, each exceptional tube can have at most one of the objects $M_i$. We can then arrange these module by taking first the preinjective modules, arranged in order of distance from the injective modules, i.e., from right-to-left, in the preinjective component, then taking the regular modules and finally the preprojective modules in decreasing distance from the projective modules, i.e., right-to-left in the preprojective component. In finite type, we simply arrange the modules right to left in the AR-quiver. This produces an exceptional sequence since extensions go right-to-left in the AR-quiver (or around the tubes).
\end{proof}


%
%

\section{Faithful group functor}

In this section we complete the proof that the cluster morphism category is $CAT(0)$ by constructing a faithful group functor to the picture group $G(\Lambda)$ for all hereditary algebras $\Lambda$ of finite or tame type. We view $G(\Lambda)$ as a groupoid with one object and we construct a faithful contravariant functor
\[
	F: \cX(\Lambda)\to G(\Lambda)
\]
This is equivalent to the faithful covariant functor $F'([X])=F([X])^{-1}$ since inversion makes every groupoid isomorphic to its opposite category. 


\subsection{Formula for the group functor}

Recall first that the picture group $G(\cW)$ of any finitely generated wide subcategory $\cW$ of $mod\text-\Lambda$ is naturally embedded as a subgroup of the picture group $G(\Lambda)$ since the generators of $G(\cW)$ are $x(\beta)$ for all dimension vectors $\beta$ of exceptional modules $M_\beta\in\cW$ and these are some of the generators of $G(\Lambda)$. Furthermore, any relations among these $x(\beta)$ are relations which hold in $G(\Lambda)$ since $\cW$ is extension closed in $mod\text-\Lambda$. Thus we have a canonical homomorphism 
\[
\varphi_\cW:G(\cW)\to G(\Lambda).
\]
Similarly, for any $\cW'\subset \cW$, we have a homomorphism $\varphi_{\cW'}^\cW:G(\cW')\to G(\cW)$ so that 
\[
\varphi_{\cW'}=\varphi_{\cW}\circ \varphi_{\cW'}^\cW:G(\cW')\to G(\Lambda).
\]
With this notation we will prove the following.

\begin{thm}\label{thm C} Let $\Lambda$ be an hereditary algebra of finite or tame type. Then, to every morphism $[X]:\cW\to \cW'$ in the cluster morphism category of $\Lambda$, there is $\gg_\cW[X]\in G(\cW)$ with the property that, for all morphisms $[Y]:\cW'\to \cW''$, we have
\[
	\gg_\cW([Y]\circ[X])=\gg_{\cW}[X]\,\varphi_{\cW'}^{\cW}(\gg_{\cW'}[Y]).
\] 
Furthermore, for any two distinct morphisms $[X]\neq [X']:\cW\to \cW'$, the corresponding group elements are distinct: $\varphi_\cW(\gg_\cW[X])\neq \varphi_\cW(\gg_\cW[X'])\in G(\Lambda)$. 
\end{thm}

In other words, $F([X])=\varphi_\cW(\gg_\cW[X])^{-1}$ is a $G(\Lambda)$ valued faithful group functor on the cluster morphism category of $\Lambda$.

This theorem, assigning a picture group element to every partial cluster tilting object of $\cW$, will follow from the special case of (complete) cluster tilting objects, the case which is already known by \cite{ITW}, \cite{IT13}, \cite{IT14} and \cite{HInogap}.

\begin{prop}\label{prop: group labels of compartments}
For any hereditary algebra $\Lambda$ of finite or tame representation type there is a function $\overline\gg$ which assigns to every cluster tilting object $T=T_1\oplus\cdots\oplus T_n$ an element $\overline\gg(T)$ in the picture group $G(\Lambda)$ satisfying the following properties.
\begin{enumerate}
\item For the cluster tilting object $P[1]$ consisting of the $n$ shifted projective modules, $\overline\gg[P[1]]=e\in G(\Lambda)$, the identity element.
\item If $T'=T\backslash T_k\oplus T_k'$ is the mutation of $T$ in direction $k$,
\[
	\overline\gg[T']=\overline\gg[T]x(\beta)^\varepsilon
\]
when $D(\beta)$ is the wall separating the two clusters, i.e., $M_\beta$ is the unique indecomposable object in $(T\backslash T_k)^\perp$. The sign $\varepsilon$ is positive when $T_k'$ is on the positive side of $D(\beta)$, i.e., when $\Hom_\Lambda(T_k',M_\beta)\neq0$, and $\varepsilon=-1$ otherwise.
\item If $T\not\cong T'$ then $\overline\gg(T)\neq \overline\gg(T')\in G(\Lambda)$.
\end{enumerate}
\end{prop}

\begin{proof}
Statement (1) is by definition and (2) is proved in \cite{IT13}, \cite{ITW}, \cite{IT14} for finite type and follows from \cite{HInogap} in the tame case since, by the polygonal deformation lemma, there is a finite sequence of mutations from any cluster to the initial cluster (of shifted projective objects) which is well-defined up to polygonal deformation and thus has well-defined picture group element associated to it by multiplying the generators $x(\beta)$ for the walls that are crossed. 

To prove (3) we need to review these arguments. We fix the clusters $T,T'$ and choose finite mutation paths from the initial cluster to $T,T'$. We keep track of all the walls of all of the clusters along those two paths. Then, we can enumerate the dimension vectors $\beta_i$ of the exceptional modules in order of length up to and including the lengths of those $\beta$ for which $D(\beta)$ occurs as a wall in one of those two mutation paths. Thus $\beta_1,\cdots,\beta_n$ are the simple roots. Let $\cS_k=\{\beta_1,\cdots,\beta_k\}$ and let $G(\cS_k)$ be the subgroup of the picture group $G(\Lambda)$ generated by $x(\beta_i)$ where $i\le k$. Then, the canonical homomorphism $G(\cS_{k})\to G(\cS_{k+1})$ has a retraction $\rho:G(\cS_{k+1})\to G(\cS_{k})$ given by sending $x(\beta_{k+1})$ to 1.

We recall \cite[Theorem 1.18]{IT14}: The walls $D(\beta_j)$ for $j\le k$ separate $\RR^n$ into convex regions called \emph{compartments} and denote $\cU_\varepsilon$ in \cite{IT14}. To each such region we can associate a picture group element $g_k(\cU_\varepsilon)\in G(\cS_k)$ so that, if two regions $\cU_\varepsilon$, $\cU_{\varepsilon'}$ are adjacent and separated by a wall $D(\beta)$, $\beta\in\cS_k$ with $\cU_{\varepsilon'}$ on the positive side of $D(\beta)$ then
\[
    g_k(\cU_{\varepsilon'})=g_k(\cU_{\varepsilon})\, x(\beta).
\]
(This is statement (2).) Furthermore, these group elements are uniquely determined if we stipulate that the component $\cU_-$ on the negative side of all the walls $D(\beta_j)$ is assigned the identity element of the group: $g_k(\cU_-)=e\in G(\cS_k)$.

For example, when $k=1$, there is only one wall $D(\beta_1)$ which is a hyperplane. This divides $\RR^n$ into two regions $\cU_-$ and $\cU_+$ with $g_1(\cU_-)=e$ and $g_1(\cU_+)=x(\beta_1)$.

When we go from $\cS_k$ to $\cS_{k+1}$ we add the wall $D(\beta_{k+1})$. This being part of a hyperplane will either cut the compartment $\cU_\epsilon$ into two pieces $\cU_{\varepsilon +}$, on the positive side of $D(\beta_{k+1})$, and $\cU_{\varepsilon -}$ on the negative side of $D(\beta_{k+1})$ or $\cU_\epsilon$ will be disjoint from $D(\beta_{k+1})$ in which case we relabel the compartment as $\cU_{\epsilon 0}$. In either case, the retraction $\rho:G(\cS_{k+1})\to G(\cS_k)$ will send the new group element $g_{k+1}(\cU_{\varepsilon\ast})$ to $g_k(\cU_\varepsilon)$ since $g_{k+1}(\cU_{\varepsilon\ast})$ is the product of $x(\beta)$ for the walls $D(\beta)$ which are crossed to get from $\cU_-$
to the compartment $\cU_{\varepsilon\ast}$ and, when the walls $D(\beta_{k+1})$ are deleted, this get the group element $\rho(g_{k+1}(\cU_{\varepsilon\ast}))=g_k(\cU_\varepsilon)$. By induction on $k$ we have the following.

\underline{Statement for $\cS_k$}:  Distinct compartments have different elements of the picture group assigned to them, i.e., $g_k(\cU_\varepsilon)\neq g_k(\cU_{\varepsilon'})$ when $\cU_\varepsilon\neq \cU_{\varepsilon'}$.

To show that this holds for $\cS_{k+1}$, suppose not. Then there are two compartments $\cU,\cU'$ in the complement of the walls $D(\beta_j)$ for $j\le k+1$ which have the same group element $g_{k+1}(\cU)=g_{k+1}(\cU')\in G(\cS_{k+1})$. Then $\rho(g_{k+1}(\cU))=\rho(g_{k+1}(\cU'))$ in $G(\cS_k)$. Therefore, by induction on $k$, $\cU,\cU'$ lie in the same compartment $\cU_\varepsilon$ for $\cS_k$. Since $\cU\neq\cU'$, by symmetry we must have $\cU=\cU_{\varepsilon+}$, $\cU'=\cU_{\varepsilon-}$. But then,
\[
    g_{k+1}(\cU)=g_{k+1}(\cU_{\varepsilon+})=g_{k+1}(\cU_{\varepsilon-})\,x(\beta_{k+1})=g_{k+1}(\cU')x(\beta_{k+1}).
\]
So, $g_{k+1}(\cU)\neq g_{k+1}(\cU')$. So, the statement holds for all $\cS_k$. Taking $\overline\gg$ to be $g_k$ for $k$ maximal we obtain statement (3). So, the proposition holds.
\end{proof}

By Proposition \ref{prop: group labels of compartments} above, the function $\overline\gg$ taking a cluster tilting object $T$ to $\overline\gg[T]\in G(\Lambda)$ is given by
\[
	\overline\gg[T]=x(\beta_1)^{\varepsilon_1}x(\beta_2)^{\varepsilon_2}\cdots x(\beta_k)^{\varepsilon_k}
\]
where $\beta_1,\cdots,\beta_k$ are the $c$-vectors of any mutation sequence from the shifted projected cluster $P_1[1],\cdots,P_n[1]$ to $T$ and the sign ${\varepsilon_i}$ is positive/negative depending on whether the mutation is green/red. (For more explanations, see \cite{I:stability}, \cite{I:MGS4CT}, \cite{IOTW3}, \cite{ITW}, \cite{IT13}.)

Recall that the \emph{picture group} of $\Lambda$ is the group with generators $x(\beta)$ for all real Schur roots $\beta$ which are also the dimension vectors of all exceptional $\Lambda$-modules. The relations are given as follows where $[a,b]:=b^{-1}aba^{-1}$.
\[
	[x_{\beta_i},x_{\beta_j}]=x_{\gamma_1}\cdots x_{\gamma_r}
\]
for $\beta_i,\beta_j$ any pair of hom orthogonal roots (roots whose corresponding modules are Hom-orthogonal) so that $\Ext(\beta_i,\beta_j)=0$ and $\gamma_k$ are the dimension vectors of indecomposable extensions $M_{\gamma_k}$ of $M_{\beta_i}^{a_k}$ by $M_{\beta_j}^{b_k}$ in increasing order of the fraction $a_k/b_k$. For example, in type $B_2$ we have:
\[
	x_\alpha x_\beta=x_\beta  x_{\alpha+\beta}x_{2\alpha+\beta}
	x_\alpha.
\]



\subsection{Proof of Theorem \ref{thm C}}

We first prove the special case when $[X]:\cW\to \cW'$ has rank 1, i.e., $X$ is either an exceptional module in $\cW$ or a shifted indecomposable projective object of $\cW$.

\begin{lem}\label{special case of Key Lemma}
Suppose that $[X]:\cW\to \cW'$ has rank $1$ and $[Y]$ is a cluster tilting set in $\cW'$. Then
\[
	\overline\gg_\cW([Y]\circ[X])=\gg_\cW[X]\varphi_{\cW'}^\cW(\overline\gg_{\cW'}[Y])
\]
\end{lem}

\begin{proof} Since finitely generated wide subcategories are equivalent to module categories, we can take $[X]$ to be a rank 1 morphism $mod\text-\Lambda\to \cW$ and $[Y]$ a cluster-tilting set in $\cW$. Since $\Lambda$ is tame, $\cW$ has finite type. Let $[Z]=\sigma_X[Y]$. Then $[X,Z]$ is a cluster-tilting set for $mod\text-\Lambda$. Given that $X$ is an indecomposable exceptional module or shifted projective module, it gives a vertex in the picture of $\Lambda$. The cluster $[X,Z]$ gives a chamber $\widetilde\cU'$ in the picture of $\Lambda$ with vertices $X$ and the elements of $[Z]$. This chamber has an associated picture group element
\[
	\overline\gg(\widetilde\cU')=\overline\gg[X,Z]\in G(\Lambda).
\]

Let $[P[1]]$ be the set of shifted projective objects of $\cW$ and let $[Q]=\sigma_X[P[1]]$. Then $[X,Q]$ is another cluster-tilting set for $\Lambda$ corresponding to another chamber $\widetilde\cU$ in the picture of $\Lambda$ and, by definition of $\gg[X]$, we have
\[
	\overline\gg(\widetilde\cU)=\overline\gg[X,Q]=\gg[X]\in G(\Lambda).
\]	
The statement of Lemma \ref{special case of Key Lemma} is
\[
	\overline\gg(\widetilde\cU')=\overline\gg(\widetilde\cU)\overline\gg_{\cW}[Y]
\]
where $\overline\gg_{\cW}[Y]\in G(\cW)$ is the picture group element assigned to the chamber $\cU'$ in the picture for $\cW$ with vertices given by $[Y]$. By definition, this is the product of all labels $x(\beta_i)$ corresponding to the walls $D(\beta_i)$ that need to be crossed to get from the chamber $\cU$ given by $[P[1]]$ to the chamber $\cU'$ given by $[Y]$. However, the picture for $\cW$ is isomorphic to the local picture for $\Lambda$ at the point $X$ and the corresponding walls have the same labels. The vertices of the picture for $\cW$ are objects of $\cW$ (or shifted projective objects) and the corresponding vertices of the picture for $\Lambda$ are given by applying $\sigma_X$ to these labels. Since $\cU,\cU'$ correspond to $\widetilde\cU,\widetilde\cU'$ by construction and the walls separating them have identical labels, $\overline\gg_\cW[Y]$ is equal to the product of the labels of the walls in the picture for $\Lambda$ separating $\widetilde\cU$ from $\widetilde\cU'$. Therefore,
\[
	\overline\gg_{\cW}[Y]=\overline\gg(\widetilde\cU)^{-1}\overline\gg(\widetilde\cU').
\]
This completes the proof of Lemma \ref{special case of Key Lemma}.

This proof is illustrated in Figure \ref{Fig:group functor on picture} in the case $X=P_2$, $[Y]=[S_3,S_1[1]]$ and $P[1]=[S_1[1],S_3[1]]$. Then $[Z]=[P_3,S_2]$ and $[Q]=[S_2,P_3[1]]$. The chamber labels $\cU,\cU',\widetilde\cU,\widetilde\cU'$ agree with the labels in the proof.
\end{proof}

Theorem \ref{thm C} will follow from the following Key Lemma.

\begin{lem}\label{Key Lemma}
Let $\cW$ be a finitely generated wide subcategory of $mod\text-\Lambda$ of rank $k$ and let $[P_\cW]=[P_1,\cdots,P_k]$ be the set of projective objects of $\cW$. Let $[P_\cW[1]]$ be the set of shifted projective objects $P_i[1]$ in the bounded derived category of $\cW$. Then, for any pair of morphisms $[X]:mod\text-\Lambda\to \cW$, $[Y]:\cW\to 0$ in the cluster morphism category we have the following equation in the picture group of $\Lambda$.
\[
	\overline\gg([X]\textstyle\coprod \sigma_X[Y])=\overline\gg([X]\textstyle\coprod \sigma_X[P_\cW[1]])\, \varphi_\cW(\overline\gg_\cW([Y])).
\]
\end{lem}

\begin{proof} This will follow from Lemma \ref{special case of Key Lemma}. Thus, suppose we have:
\[
	mod\text-\Lambda \xrightarrow{[X]}\cW\xrightarrow{[Y]}0
\]
and $[P[1]]:\cW\to 0$ is the cluster morphism given by the shifted projective objects in $\cW$. We need to show that 
\begin{equation}\label{eq: key lemma}
	\overline\gg([Y]\circ[X])=\overline\gg[X\textstyle\coprod \sigma_X(P[1])]\varphi_{\cW}(\overline\gg_{\cW}[Y]).
\end{equation}
When $[X]$ has rank 1, this holds by Lemma \ref{special case of Key Lemma}. So, suppose $[X]$ has rank $\ge2$. Then $[X]$ factors as a composition
\[
	[X]=[B]\circ [A]:mod\text-\Lambda\xrightarrow{[A]} \cW^\ast\xrightarrow{[B]} \cW.
\]
Since $[A],[B]$ have smaller rank than $[X]$ we have
\[
    \overline\gg([Y]\circ[X])=\overline\gg([Y]\circ[B]\circ [A])=
    \gg[A]\overline\gg_{\cW^\ast}([Y]\circ[B])=\gg[A]\gg_{\cW^\ast}[B]\overline\gg_{\cW'}[Y].
\]
In the special case $Y=P[1]$ we get:
\[
    \gg[X]:=\overline\gg([P[1]]\circ[X])=
    \gg[A]\overline\gg_{\cW^\ast}([P[1]]\circ[B])=\gg[A]\gg_{\cW^\ast}[B].
\]
So, 
\[
    \overline\gg([Y]\circ[X])=\gg[A]\gg_{\cW^\ast}[B]\overline\gg_{\cW'}[Y]=\gg[X]\overline\gg_{\cW'}[Y].
\]
This proves the Key Lemma \ref{Key Lemma}. 
\end{proof}

We need one more lemma to obtain a faithful group functor for the cluster morphism category of $\Lambda$ or, equivalently, for any finitely generated wide subcategory $\cW$ of finite or tame type in $mod\text-H$ for any hereditary algebra $H$. Indeed, for any such wide subcategory $\cW$ and any partial cluster tilting object $X$ in $\cW$ we define $\gg_\cW([X])\in G(\cW)$ to be
\begin{equation}\label{def of g(X)}
	\gg_\cW([X]):=\overline\gg_\cW([X]\textstyle\coprod \sigma_X[P[1]])
\end{equation}
where $P[1]$ is the set of all shifted projective objects in the wide subcategory $\cW'=X^\perp\cap \cW$ of $\cW$. Then, as a special case of Lemma \ref{Key Lemma} we have
\[
	\overline\gg_\cW([Y]\circ[X])=\overline\gg_\cW([X]\textstyle\coprod \sigma_X[Y])=\gg_\cW([X])\varphi_{\cW'}^\cW(\overline\gg_{\cW'}([Y]))
\]
for any cluster tilting set $Y$ in $\cW'$.

\begin{lem}\label{lem: g(X) is a group functor}
Let $[X]:\cW\to \cW'$, $[Y]:\cW'\to \cW''$ be morphisms in the cluster morphism category of $\Lambda$. 
Then
\[
	\gg_\cW([X]\textstyle\coprod \sigma_X[Y])=\gg_\cW([X])\varphi_{\cW'}^\cW(\gg_{\cW'}([Y])).
\]
\end{lem}

\begin{proof}
Let $Z$ be any cluster tilting set in $\cW''$ and let $[A]=[Y]\circ [X]$, $[B]=[Z]\circ [Y]$, $[C]=[Z]\circ [A]=[B]\circ[X]$. Since $[B]=[Z]\circ [Y]$ is a cluster tilting set for $\cW'$ we have
\[
	\overline\gg_{\cW'}([B])=\gg_{\cW'}([Y])\overline\gg_{\cW''}([Z])
\]
Since $[C]=[B]\circ [X]=[Z]\circ [A]$ is a cluster tilting set for $\cW$ we also have
\[
	\overline\gg_\cW([C])=\overline\gg_\cW([B]\circ [X])=\gg_\cW([X])\overline\gg_{\cW'}([B])=\gg_\cW([X])\gg_{\cW'}([Y])\overline\gg_{\cW''}([Z])
\]
\[
	\overline\gg_\cW([C])=\overline\gg_\cW([Z]\circ [A])=\gg_{\cW'}([A])\overline\gg_{\cW''}([Z])
\]
Multiplying by $\overline\gg_{\cW''}([Z])^{-1}$ we obtain
\[
	\gg_\cW([X])\gg_{\cW'}([Y])=\gg_{\cW'}([A])=\gg_{\cW'}([X]\textstyle\coprod \sigma_X[Y])
\]
as claimed.
\end{proof}

\begin{eg}
Consider the standard example: $\Lambda=KA_3$, for the linearly oriented quiver $A_3$: 
\[
	1\leftarrow 2\leftarrow 3.
\]
Consider the cluster morphisms
\[
	mod\text-\Lambda\xrightarrow{[X]} \cW\xrightarrow{[Y]} \cW'
\]
where $X=P_2$, with $\cW=P_2^\perp$ having indecomposable objects $S_1$ and $S_3=Y$. Then $\sigma_{P_2}:\cC_\cW\to \cC_\Lambda$ sends the objects $S_1,S_3,S_1[1],S_3[1]$ to $S_1, P_3, S_2, P_3[1]$ resp. So, $\sigma_X(Y)=\sigma_{P_2}(S_3)=P_3$ which makes $\cW'=(P_2\oplus P_3)^\perp$ with one indecomposable object $S_1$. 

By \eqref{def of g(X)}, the picture group element corresponding to $X=P_2$ is
\[
    \gg(P_2)=\overline\gg(P_2 \textstyle\coprod \sigma_{P_2}(S_1[1],S_3[1]))=\overline\gg(P_2,S_2,P_3[1])=x(S_2)x(P_2)\in G(\Lambda)
\]
since, to get to the chamber $\widetilde \cU$ in Figure \ref{Fig:group functor on picture} with corners $P_2,S_2,P_3[1]$ from the unbounded chamber, we need to pass through the walls $D(S_2)$ and $D(P_2)$.

\begin{figure}[htbp]
\begin{center}
\begin{tikzpicture}
\begin{scope}[yscale=.75]
	\clip (-5.75,-4.5) rectangle (4.25,4.8);
		\draw[very thick] (.75,1.3) circle [radius=3cm];
		\draw[very thick,red] (-2.25,1.3) circle [radius=3cm];
		\draw[very thick,green] (-.75,-1.3) circle [radius=3cm];
		\begin{scope}
		\clip (-.75,-5.2) rectangle (4.25,5);
		\draw[very thick] (-.75,1.3) ellipse [x radius=3cm,y radius=2.6cm];
		\end{scope}
		\begin{scope}[rotate=60]
		\clip (0,-5) rectangle (-5,5);
		\draw[very thick] (0,0) ellipse [x radius=3cm,y radius=2.6cm];
		\end{scope}
\begin{scope}
\clip (-2.4,-3) rectangle (2,.1);
		\draw[very thick] (-.75,.43)  circle [radius=2.68cm];
		\end{scope}
\end{scope}
\draw[fill,blue] (-.73,-.97) circle[radius=1mm]; 
\draw[blue] (-.73,-.93)node[above]{$P_3$};
\draw[fill,blue] (.73,.97) circle[radius=1mm]; 
\draw[blue] (.73,.8)node[left]{$P_2$};
\draw[fill,blue] (1.9,.03) circle[radius=1mm]; 
\draw[blue] (1.9,.1)node[right]{$S_2$};
\draw[fill,blue] (-2.25,.97) circle[radius=1mm]; 
\draw[blue] (-2.2,.8)node[right]{$P_1=S_1$};

\draw[fill,blue] (-.73,2.9) circle[radius=1mm]; 
\draw[blue] (-.73,2.8)node[below]{$P_3[1]$};
\draw (2.2,2.1)node{\tiny$ D(P_2)$};
\draw (4,2)node{\tiny$D(S_2)$};
\draw[green] (1.5,.8) node{\tiny$D(S_3)$};
\draw[red] (.6,2.1) node{\tiny$D(S_1)$};
\draw (1.3,1.5) node{$\widetilde\cU$};
\draw (1,0) node{$\widetilde\cU'$};
\draw[thick,blue,->] (5.4,0)--(4.4,0) ;
\draw[blue] (5,-.4) node{$\sigma_{P_2}$};
\begin{scope}[xshift=7cm,yshift=-1cm]
\draw (1,-1.6) node{$\cW=P_2^\perp$};
\begin{scope}
\clip (.73,1.2) circle[radius=2cm];
		\draw[very thick,red] (-2.25,1.3) circle [radius=3cm];
		\draw[very thick,green] (-.75,-1.3) circle [radius=3cm];
\end{scope}
\draw[green] (1.4,.8) node[right]{\tiny$D(S_3)$};
\draw[red] (.6,2.1) node[right]{\tiny$D(S_1)$};
\draw[fill,blue] (.73,1.3) circle[radius=1mm]; 
\draw[blue] (.73,1.4) node[right]{$0$}; 
\draw[fill,blue] (1.94,0) circle[radius=1mm]; 
\draw[blue] (2,0) node[right]{$S_1[1]$}; 
\draw[fill,blue] (.45,0) circle[radius=1mm]; 
\draw[blue] (.4,0) node[left]{$S_3$}; 
\draw[fill,blue] (.3,2.9) circle[radius=1mm]; 
\draw[blue] (.4,3) node[right]{$S_3[1]$}; 
\draw[fill,blue] (-.8,1.7) circle[radius=1mm]; 
\draw[blue] (-.8,1.7) node[above]{$S_1$}; 
\draw (2,1.7) node{$\cU$};
\draw (1.2,-.2) node{$\cU'$};
\end{scope}
%
\end{tikzpicture}
\caption{In the picture for $A_3$ on the left, the chamber $\widetilde\cU$ with corners $P_2, S_2, P_3[1]$ has picture group element $\gg(P_2)=\overline\gg(P_2, S_2, P_3[1])=x(S_2)x(P_2)$. In the picture for $\cW=P_2^\perp$ on the right, the corresponding chamber $\cU$ has identity group element $\overline\gg_\cW(S_1[1],S_2[1])=e$. Crossing the wall $\color{green}D(S_3)$, both picture group elements are multiplied by $x(S_3)$.}
\label{Fig:group functor on picture}
\end{center}
\end{figure}
Since $\cW$ is semi-simple, $G(\cW)=\ZZ^2$ is the free abelian group generated by $x(S^1),x(S^3)$. The picture group element for $Y=S_3$ is
\[
    \gg_\cW(S_3)=\overline\gg_\cW(S_3\textstyle\coprod \sigma_{S_3}(S_1[1]))=\overline\gg_\cW(S_3, S_1[1])=x(S_3).
\]
Since $\sigma_{P_2,P_3}(S_1[1])=\sigma_{P_2}(S_1[1])=S_2$, the product of these is
\[
    \gg(P_2, P_3)=\overline\gg(P_2,P_3\textstyle\coprod \sigma_{P_2,P_3}(S_1[1]))=\overline\gg(P_2,P_3,S_2)=x(S_2)x(P_2)x(S_3)=
    \gg(P_2)\gg_\cW(S_3)
\]
Indeed, one must pass through the walls $D(S_2),D(P_2),D(S_3)$ to get to the chamber $\widetilde \cU'$ in Figure \ref{Fig:group functor on picture} with corners $P_2,P_3,S_2$.
\end{eg}

\begin{thm}\label{thm: faithful group functors exist}
Equation \eqref{def of g(X)} gives a faithful (contravariant) group functor from the cluster morphism category of $\Lambda$ to its picture group $G(\Lambda)$ for $\Lambda$ a hereditary algebra of finite or tame representation type.
\end{thm}

\begin{proof}
Lemma \ref{lem: g(X) is a group functor} shows that $F([X])=\gg_\cW([X])^{-1}$ is a group functor on the cluster morphism category of $\Lambda$. It remains to show that it is faithful. To show this, suppose not. Then there are cluster morphisms $[X],[Y]:\cW\to \cW'$ with the same group value $\gg_\cW[X]=\gg_\cW[Y]$. This would imply that the cluster tilting objects $[X]\oplus P[1]$ and $[Y]\oplus P[1]$ in $\cW$ have the same group value in $G(\cW)$ which is not possible by Proposition \ref{prop: group labels of compartments}.
\end{proof}

We now prove our main theorem.

\begin{thm}[\ref{thm B}]\label{main thm}
The cluster morphism category of an hereditary algebra of finite or tame representation type is a $CAT(0)$-category if and only if there are no tubes of rank $\ge3$ in its Auslander-Reiten quiver.
\end{thm}

\begin{proof}
Recall that a $CAT(0)$-category is a cubical category whose classifying space is locally $CAT(0)$ and thus a $K(\pi,1)$ by \cite{Gromov87}. In Theorem \ref{thm: CMC is cubical} we showed that the cluster morphism category $\cX(\Lambda)$ of any finite dimensional hereditary algebra satisfies the definition of a cubical category (Definition \ref{def: cubical category}).

Theorem \ref{thm: conditions for CAT0} gives criteria for when a cubical category is $CAT(0)$. The first two conditions are necessary and the three conditions together are sufficient. The first condition is immediate since it requires that cluster-tilting objects be given by a pairwise compatibility condition which is true by definition for any hereditary algebra. Theorem \ref{old thm B} shows that the second condition (that the $c$-vectors of a cluster-tilting object are given by a pairwise compatibility condition) fails for tame algebras having an exceptional tube of rank $\ge3$ and holds for all other $\Lambda$ of finite or tame representation type. Finally, Theorem \ref{thm: faithful group functors exist} proves the last condition, that there exists a faithful group functor from the cluster morphism category to some group when $\Lambda$ has finite or tame representation type.

This concludes the proof of our main theorem, Theorem \ref{thm B}.
\end{proof}


\end{document}